\documentclass[a4paper, 12pt]{amsart}
\usepackage{amsfonts, amsthm, amssymb, amsmath}
\usepackage{mathrsfs,array}
\usepackage{xy}
\usepackage{hyperref}
\usepackage{verbatim}
\usepackage[latin1]{inputenc}
\input xy
\xyoption{all}

\newtheorem{theo}{Theorem}[section]
\newtheorem{prop}[theo]{Proposition}
\newtheorem{defi}[theo]{Definition}
\newtheorem{lemm}[theo]{Lemma}

\newtheorem{conj}[theo]{Conjecture}

\newcommand{\mb}{\mathbb}
\newcommand{\mc}{\mathcal}
\newcommand{\mf}{\mathfrak}

\newcommand{\ra}{\rightarrow}

\DeclareMathOperator{\Spa}{Spa}

\DeclareMathOperator{\Gal}{Gal}

\DeclareMathOperator{\Frob}{Frob}
\DeclareMathOperator{\Ind}{Ind}

\DeclareMathOperator{\GL}{GL}

\DeclareMathOperator{\Rep}{Rep}

\DeclareMathOperator{\ord}{ord}
\DeclareMathOperator{\sss}{ss}
\DeclareMathOperator{\HT}{HT}

\title{On non-abelian Lubin-Tate theory and analytic cohomology}
\date {\today}
\author{Przemys\l aw Chojecki }
\email{chojecki@math.jussieu.fr}

\begin{document}

\begin{abstract} We prove that the $p$-adic local Langlands correspondence for $\GL_2(\mb{Q}_p)$ appears in the \'etale cohomology of the Lubin-Tate tower at infinity. We use global methods using recent results of Emerton on the local-global compatibility and hence our proof applies to local Galois representations which come via a restriction from a global pro-modular Galois representations. We also discuss a folklore conjecture which states that the $p$-adic local Langlands correspondence appears in the de Rham cohomology of the Lubin-Tate tower (Drinfeld tower). We show that a study of the de Rham cohomology for perfectoid spaces reduces to a study of the analytic cohomology and we state a natural conjecture related to it.
\end{abstract}

\maketitle
\tableofcontents

\section{Introduction}

\noindent The $p$-adic Langlands program, started by Christophe Breuil and developed largely by Laurent Berger, Pierre Colmez, Matthew Emerton and Mark Kisin, has as a goal to establish a correspondence between $p$-adic Galois representations and representations of $p$-adic reductive groups on $p$-adic Banach spaces. It has (and will have) many applications, for example the Fontaine-Mazur conjecture (see \cite{em2}). Unfortunately, at the moment the $p$-adic correspondence is constructed (mostly by Colmez) only for $\GL_2(\mb{Q}_p)$ and it seems hard to generalize it to other groups because of many algebraic obstacles. 

\medskip

\noindent Hence it seems natural to try to find the correspondence in some appropiate cohomology groups as was done for the classical Langlands correspondence. We are interested in the $p$-adic completed and analytic cohomologies of Shimura varieties and Rapoport-Zink spaces, which are natural objects to consider in the context of the Langlands program. This paper might be seen as a sequel to \cite{cho}, where we have studied the mod $p$ \'etale cohomology of the Lubin-Tate tower. Here we turn to the study of the $p$-adic completed and analytic cohomologies.

\medskip

\noindent There are two goals which we want to accomplish in this short paper. The first one is to show a result analogous to the one obtained in \cite{cho}, namely to show that the $p$-adic local Langlands correspondence for $\GL_2(\mb{Q}_p)$ appears in the \'etale cohomology of the Lubin-Tate tower at infinity. The methods we use are partly those of \cite{cho} (localisation at a supersingular representation; use of the local-global compatibility of Emerton), though we approach them differently by working in the setting of adic spaces (we have worked with Berkovich spaces in \cite{cho}). This gives us more freedom as we can work directly at the infinite level (modular curves at the infinite level; Lubin-Tate tower at the infinite level) thanks to the work of Scholze on perfectoid spaces (\cite{sch1}, \cite{sch3}, \cite{sw}). In this way, we do not need anymore to pass to the limit in the cohomology, as working at the infinite level is the same as working with the completed cohomology (see Chapter IV of \cite{sch3} for torsion coefficients). We prove our main result (Theorem 4.3) for local Galois representations $\rho _p$ which are restrictions of some global pro-modular (a notion from \cite{em2}) representations $\rho$ and such that the mod $p$ reduction $\bar{\rho}_p$ is absolutely irreducible. We need these assumptions in order to be able to use the main result of \cite{em2}.

\medskip

\noindent The second goal of this text is to discuss the folklore conjecture which roughly states that the $p$-adic local Langlands correspondence appears in the de Rham cohomology of the Drinfeld tower. As far as we know, this conjecture is not stated anywhere explicitely in the literature, though there was some work done towards it. The reader should consult \cite{schr} for some partial progress at the 0-th level of the tower. Thanks to the work of Scholze-Weinstein (\cite{sw}) we can work directly at the infinite level which we do. Moreover, because of the duality of Rapoport-Zink spaces at the infinite level (which goes back to Faltings; see Section 7 of \cite{sw}), we know that the Drinfeld space at infinity $\mc{M}_{Dr, \infty}$ is isomorphic to the Lubin-Tate space at infinity $\mc{M} _{LT, \infty}$ and hence we can consider only the Lubin-Tate tower which is easier to relate to modular curves.

\medskip

\noindent As to the folklore conjecture, we give a short argument at the beginning of Section 4, which explains why the de Rham cohomology of $\mc{M} _{LT, \infty}$ simplifies greatly. The reason is that for any perfectoid space $X$ (hence for $\mc{M} _{LT, \infty}$ after \cite{sw}) the cohomology groups of $j$-th differentials $H^i(X, \Omega ^j _X)$ vanishes for $j>0$ and any $i$. This reduces the study of the de Rham cohomology to the study of the cohomology of the structure sheaf (which we refer to as the analytic cohomology - with topology defined by open subsets) which should be a good substitute for the de Rham cohomology in the setting of perfectoid spaces. We state the folklore conjecture for the analytic cohomology in the last section.

\medskip

\noindent At the end we remark that one problem with the de Rham cohomology for perfectoid spaces, if one would like to define it in some meaningful way, is the lack of finiteness result. We should mention the work of Cais (\cite{cai}), where the author consider integral structures on the de Rham cohomology of curves. The aim is to $p$-adically complete the de Rham cohomology of the tower of modular curves, as was done with the \'etale cohomology by Emerton (\cite{em1}). It seems interesting to determine what one would get by applying his construction at each finite level and then passing to the limit and how it would relate to the de Rham cohomology of the modular curve at the infinite level.

\medskip

\noindent \textbf{Acknowledgements.} I thank Christophe Breuil, Jean-Francois Dat, David Hansen, Peter Scholze and Jared Weinstein for useful discussions and correspondence related to this text.

\section{Modular curves at infinity}

In this section we review the geometric background which we use. We describe modular curves (and their compactifications) at the infinite level and we deal with the ordinary locus and the supersingular locus. We will use the language of adic spaces for which the reader should consult \cite{hu} and \cite{sch1}.

\medskip

We let $E$ be a finite extension of $\mb{Q}_p$ with the ring of integers $\mc{O}$ and the residue field $k = \mc{O} / \varpi $ where $\varpi$ is a uniformiser. This is our coefficient field.

\subsection{Geometry of modular curves}

We denote open modular curves over $\mb{C}$ for an open compact subset $K \subset \GL_2(\mb{A}_f)$ by
$$Y(K) = \GL_2(\mb{Q}) \backslash (\mb{C} \backslash \mb{R}) \times \GL_2(\mb{A}_f) / K$$ 
There is a canonical algebraic model of it over $\mb{Q}$. We fix some complete and algebraically closed extension $C$ of $\mb{Q}_p$. Let $\mc{O}_C$ be the ring of integers of $C$. We consider modular curves as adic spaces over $\Spa (C, \mc{O}_{C})$ which we may do after base-changing each $Y(K)$.

We let $X(K)$ be the compactification of $Y(K)$, which we also consider as an adic space over $\Spa(C, \mc{O}_{C})$. We will work with modular curves at the infinite level. We recall Scholze's results. We use $\sim$ in the sense of Definition 2.4.1 in \cite{sw}.

\begin{theo} For any sufficiently small level $K^p \subset \GL_2(\mb{A}_f^p)$ there exist adic spaces $Y(K^p)$ and $X(K^p)$ over $\Spa(C, \mc{O}_{C})$ such that
$$Y(K^p) \sim \varprojlim _{K_p} Y(K_pK^p)$$
$$X(K^p) \sim \varprojlim _{K_p} X(K_pK^p)$$
where $K_p$ runs over open compact subgroups of $\GL_2(\mb{Q}_p)$.
\end{theo}
\begin{proof}
See Theorem III.1.2 in \cite{sch3}. 
\end{proof}

In what follows we will write $Y= Y(K^p)$ and $X=X(K^p)$, having fixed one tame level $K^p$ throughout the text. 

\medskip


For the maximal compact open subgroup $\GL_2(\mb{Z}_p)$ we can define the supersingular locus $Y(\GL_2(\mb{Z}_p)K^p)_{\sss}$ (respectively, the ordinary locus $Y(\GL_2(\mb{Z}_p)K^p)_{\ord}$) as the inverse image under the reduction of the set of supersingular points (resp. closure of the inverse image of the ordinary locus) in the special fiber of $Y(\GL_2(\mb{Z}_p) K^p)$. Then for any compact open subgroup $K_p \subset \GL_2(\mb{Z}_p)$, we define $Y(K_pK^p)_{\sss}$ (resp. $Y(K_pK^p)_{\ord}$) as the pullback of $Y(\GL_2(\mb{Z}_p)K^p)_{\sss}$ (resp. $Y(\GL_2(\mb{Z}_p)K^p)_{\ord}$). Hence $Y(K_pK^p)_{\ord}$ is the complement of $Y(K_pK^p)_{\sss}$ and hence a closed subspace of $Y(K_pK^p)$. We define similarly the supersingular locus $X(K_pK^p)_{\sss}$ and the ordinary locus $X(K_pK^p)_{\ord}$ of $X(K_pK^p)$. Using the pullback from the finite level, we define also $X_{\sss},Y_{\sss},X_{\ord},Y_{\ord}$ at the infinite level. The reader may consult the discussion in \cite{sch3} which appears after Theorem III.1.2.


\begin{theo}\label{locus-infty}
There exist adic spaces $Y _{\sss}$, $Y_{\ord}$ and $X _{\sss}$, $X_{\ord}$ over $\Spa(C, \mc{O}_{C})$ such that
$$Y _{\sss} \sim  \varprojlim _{K_p} Y(K_p K^p) _{\sss}$$
$$Y _{\ord} \sim \varprojlim _{K_p} Y(K_p K^p) _{\ord}$$ 
and similarly for $X_{\sss}$ and $X_{\ord}$. Here $K_p$ runs over open compact subgroups of $\GL_2(\mb{Q}_p)$.
\end{theo}
\begin{proof} Follows from Proposition 2.4.3 in \cite{sw}.
\end{proof}

One of the main results of \cite{sch3} (Theorem III.1.2), is the construction of the Hodge-Tate period map $\pi _{\HT}$ which is a $\GL_2(\mb{Q}_p)$-equivariant morphism
$$\pi _{\HT} : X \ra (\mb{P} ^1)^{ad}$$
where $(\mb{P} ^1)^{ad}$ is the adic projective line over $\Spa(C, \mc{O}_{C})$. This morphism commutes with Hecke operators away from $p$ for the trivial action of these Hecke operators on $(\mb{P}^1)^{ad}$. Moreover, the decomposition of $X$ into the supersingular and the ordinary locus can be seen at the flag variety level. Namely, we have (see the discussion after Theorem III.1.2 in \cite{sch3})
$$X_{\ord} = \pi ^{-1} _{\HT} (\mb{P}^1 (\mb{Q}_p))$$
$$X_{\sss} = \pi ^{-1} _{\HT} ((\mb{P}^1)^{ad} \backslash \mb{P}^1 (\mb{Q}_p))$$

\medskip

We let 
$$j: X _{\sss} \hookrightarrow X$$
denote the open immersion and we put
$$i: X _{\ord} \ra X$$
For any injective \'etale sheaf $I$ on $X$ we have an exact sequence of global sections
$$0\ra \Gamma _{X_{\ord}}(X, I) \ra \Gamma(X,I) \ra \Gamma (X_{\sss}, j^*I) \ra 0$$
which gives rise to the exact sequence of \'etale cohomology for any \'etale sheaf $F$ on $X$ (take an injective resolution $I^{\bullet}$ of $F$ and apply the above exact sequence to it) 
$$... \ra H^0 (X_{\sss}, j^* F) \ra H^1 _{X_{\ord}} (X, F) \ra H^1 (X, F) \ra H^1 (X_{\sss}, j^* F) \ra  ...$$
By specialising $F$ to a constant sheaf $\mc{O} / \varpi ^s \mc{O}$ ($s>0$) we get an exact sequence
$$... \ra H^0 (X_{\sss}, \mc{O} / \varpi ^s \mc{O}) \ra H^1 _{X_{\ord}} (X, \mc{O} / \varpi ^s \mc{O}) \ra H^1 (X, \mc{O} / \varpi ^s \mc{O}) \ra H^1 (X_{\sss}, \mc{O} / \varpi ^s \mc{O}) \ra  ...$$
We can obtain an analogous exact sequence for analytic cohomology which we review later. In what follows we will be interested in the $p$-adic completed cohomology, introduced by Emerton in \cite{em1}. We define
$$H^i (X, E) = \left( \varprojlim _{s} H^i _{et}(X,\mc{O} / \varpi ^s \mc{O}) \right) \otimes _{\mc{O}} E$$
Using the fact that $X \sim \varprojlim _{K_p} X(K_p K^p)$ and Theorem 7.17 in \cite{sch1}, we have
$$H^i (X, E) = \left( \varprojlim _{s} \varinjlim _{K_p} H^i _{et}(X(K_pK^p), \mc{O} / \varpi ^s \mc{O}) \right) \otimes _{\mc{O}} E$$ 
which is precisely the $p$-adic completed cohomology of Emerton. We use similar definitions for $X_{\sss}$ and $X_{\ord}$. 




\subsection{Ordinary locus}

We recall the decomposition of the ordinary locus, which implies that representations arising from the cohomology are induced from a Borel subgroup. This is a classical and well-known result, but we shall give it a short proof using recent results of Scholze and the fact that we are working at the infinite level. We have given a different proof in Section 2.2 of \cite{cho}.

\begin{prop}
The \'etale (and also analytic) cohomology of $X_{\ord}$ is induced from a Borel subgroup $B(\mb{Q}_p)$ of upper-triangular matrices in $\GL_2(\mb{Q}_p)$
$$H^i _{X_{\ord}}(X, F) = \Ind  _{B(\mb{Q}_p)} ^{\GL_2(\mb{Q}_p)} W(F)$$
where $F= \mc{O}/\varpi ^s \mc{O}$ is an \'etale constant sheaf on $X_{\ord}$ and $W(F)$ is a certain cohomology space defined below in the proof which depends on $F$ and admits an action of $B(\mb{Q}_p)$. The induction appearing above is the smooth induction.
\end{prop}
\begin{proof}

Recall that $X_{\ord} = \pi ^{-1} _{\HT} (\mb{P}^1 (\mb{Q}_p))$, where $\pi _{\HT}$ is the Hodge-Tate period map. Let $\infty = (\begin{smallmatrix}1 \\ 0 \end{smallmatrix}) \in \mb{P}^1(\mb{Q}_p)$. The stabilizer of $\infty$ is equal to the Borel subgroup $B(\mb{Q}_p)$ of upper-triangular matrices in $\GL_2(\mb{Q}_p)$. We have
$$H^i _{X_{\ord}}(X, F) = H^i(X_{\ord}, i^!F)=H^0(\mb{P}^1(\mb{Q}_p), R^i \pi _{\HT, *}(i^!F)) = \Ind  _{B(\mb{Q}_p)} ^{\GL_2(\mb{Q}_p)} H^0(\{\infty \}, R^i \pi _{\HT, *}(i^!F))$$
where the second isomorphism follows from the continuity of $\pi _{\HT}$. Those are all smooth spaces, because $H^i(X_{\ord}, i^! F)$ is smooth (by Theorem \ref{locus-infty} and recalling that $F=\mc{O}/\varpi ^s \mc{O}$).

\end{proof}

\subsection{Supersingular locus}

Let us denote by $\mc{M} _{LT, K_p}$ the Lubin-Tate space for $\GL_2(\mb{Q}_p)$ at the level $K_p$, where $K_p$ is a compact open subgroup of $\GL_2(\mb{Q}_p)$. See Section 6 of \cite{sw} for a definition. We just recall that this is a deformation space for $p$-divisible groups with an additional data and it is a local analogue of modular curves. We view it as an adic space over $\Spa (C, \mc{O}_{C})$.

Once again, we would like to pass to the limit and work with the space at infinity. 

\begin{theo}
There exists a perfectoid space $\mc{M} _{LT, \infty}$ over $\Spa (C, \mc{O}_{C})$ such that
$$\mc{M} _{LT, \infty} \sim \varprojlim _{K_p} \mc{M} _{LT, K_p}$$
where $K_p$ runs over compact open subgroups of $\GL_2(\mb{Q}_p)$.
\end{theo}
\begin{proof} This is Theorem 6.3.4 from \cite{sw}. One defines $\mc{M} _{LT,\infty}$ as a deformation functor of $p$-divisible groups with a trivialization of Tate modules.
\end{proof}

To compare $X$ and $\mc{M}_{LT, \infty}$ (hence their cohomology groups) we use the $p$-adic uniformisation of Rapoport-Zink at the infinite level. Let us denote by $D$ the quaternion algebra over $\mb{Q}$ which is ramified exactly at $p$ and $\infty$. The $p$-adic uniformisation of Rapoport-Zink states
\begin{prop}
We have an isomorphism of adic spaces
$$X_{\sss} \sim \varprojlim _{K_p}   D^{\times}(\mb{Q}) \backslash \left( \mc{M} _{LT, K_p} \times  \GL_2(\mb{A}_f ^p) \right) / K_pK^p $$
This isomorphism is equivariant with respect to the action of the Hecke algebra of level $K^p$.
\end{prop}
\begin{proof}
The uniformisation at finite level is proved in \cite{rz}. We adify their construction and pass to the limit using Theorem \ref{locus-infty}.
\end{proof}

\section{On admissible representations}

Having recalled the geometric results, we now pass to the results about representations of $\GL_2(\mb{Q}_p)$. We review and prove some facts about Banach admissible representations. Then we recall recent results of Paskunas which allow us to consider the localisation functor. 

\subsection{General facts and definitions} We start with general facts about admissible representations. In our definitions, we will follow \cite{em4}. As before, let $E$ be a finite extension of $\mb{Q}_p$ with ring of integers $\mc{O}$, a uniformiser $\varpi$ and the residue field $k$. Let $C(\mc{O})$ denote the category of complete Noetherian local $\mc{O}$-algebras having finite residue fields. Let us consider $A \in C(\mc{O})$. We let $G$ be any connected reductive group over $\mb{Q}_p$.

\begin{defi} Let $V$ be a representation of $G$ over $A$. A vector $v \in V$ is smooth if $v$ is fixed by some open subgroup of $G$ and $v$ is annihilated by some power $\mf{m} ^i$ of the maximal ideal of $A$. Let $V_{sm}$ denote the subset of smooth vectors of $V$. We say that a $G$-representation $V$ over $A$ is smooth if $V=V_{sm}$.
\medskip

\noindent A smooth $G$-representation $V$ over $A$ is admissible if $V^H[\mf{m} ^i]$ (the $\mf{m} ^i$-torsion part of the subspace of $H$-fixed vectors in $V$) is finitely generated over $A$ for every open compact subgroup $H$ of $G$ and every $i \geq 0$.
\end{defi}
\begin{defi} We say that a $G$-representation $V$ over $A$ is $\varpi$-adically continuous if $V$ is $\varpi$-adically separated and complete, $V[\varpi ^{\infty}]$ is of bounded exponent, $V/\varpi ^i V$ is a smooth $G$-representation for any $i \geq 0$.
\end{defi}
\begin{defi}
A $\varpi$-adically admissible representation of $G$ over $A$ is a $\varpi$-adically continuous representation $V$ of $G$ over $A$ such that the induced $G$-representation on $(V/\varpi V)[\mf{m}]$ is admissible smooth over $A/\mf{m}$.
\end{defi} 
This definition implies that for every $i \geq 0$, the $G$-representation $V/\varpi ^i V$ is smooth admissible. See Remark 2.4.8 in \cite{em4}.
\begin{defi} We call a $G$-representation $V$ over $E$ Banach admissible if there exists a $G$-invariant lattice $V^{\circ} \subset V$ over $\mc{O}$ such that $V^{\circ}$ is $\varpi$-adically admissible as a representation of $G$ over $\mc{O}$.
\end{defi} 
\begin{prop} The category of $\varpi$-adically admissible representations of $G$ over $A$ is abelian and moreover, a Serre subcategory of the category of $\varpi$-adically continuous representations.
\end{prop}
\begin{proof} The category is anti-equivalent to the category of finitely generated augmented modules over certain completed group rings. See Proposition 2.4.11 in \cite{em4}.
\end{proof}
Now, we will prove an analogue of Lemma 13.2.3 from \cite{bo} in the $l=p$ setting. We will later apply this lemma to the cohomology of the ordinary locus to force its vanishing after localisation at a supersingular representation of $\GL_2(\mb{Q}_p)$. We have proved it already in the mod $p$ setting as Lemma 3.3 in \cite{cho}.
\begin{lemm}
For any smooth admissible representation $(\pi,V)$ of the parabolic subgroup $P \subset G$ over $A$, the unipotent radical $U$ of $P$ acts trivially on $V$.
\end{lemm}
\begin{proof} Let $L$ be a Levi subgroup of $P$, so that $P=LU$. Let $v \in V$ and let $K_P = K_L K_U$ be a compact open subgroup of $P$ such that $v\in V ^{K_P}$. We choose an element $z$ in the centre of $L$ such that:
$$z^{-n}K_P z^n \subset ... \subset z^{-1}K_P z \subset K_P \subset z K_P z^{-1} \subset ... \subset z^n K_P z^{-n} \subset ...$$
and $\bigcup _{n \geq 0} z^n K_P z^{-n} = K_L U$. For every $n$ and $m$, modules $V^{z^{-n}K_P z^n}[\mf{m} ^i]$ and $V^{z^{-m}K_P z^m}[\mf{m} ^i]$ are of the same length  for every $i \geq 0$, as they are isomorphic via $\pi(z ^{n-m})$. We naturally have an inclusion $V^{z^{-n}K_P z^n}[\mf{m} ^i] \subset V^{z^{-m}K_P z^m}[\mf{m} ^i]$ and hence we get an equality $V^{z^{-n}K_P z^n}[\mf{m} ^i] = V^{z^{-m}K_P z^m}[\mf{m} ^i]$. By smoothness, there exists $i$ such that $v \in V[\mf{m} ^i]$. Thus we have $v \in V^{K_P}[\mf{m} ^i] = V^{z^{-n}K_P z^n}[\mf{m} ^i]=V^{K_L U}[\mf{m} ^i]$ which is contained in $V^U[\mf{m} ^i]$.  
\end{proof}
\begin{lemm} For any $\varpi$-adically admissible representation $(\pi,V)$ of the parabolic subgroup $P \subset G$ over $A$, the unipotent radical $U$ of $P$ acts trivially on $V$.
\end{lemm}
\begin{proof} By the remark above, each $V/ \varpi ^i V$ is admissible, and hence the preceding lemma applies, so that $U$ acts trivially on each $V/ \varpi ^i V$. But $V = \varprojlim _i V/ \varpi ^i V$, hence $U$ acts trivially on $V$. 
\end{proof}

Later on, we will need the following result.
\begin{lemm} Let $V = \Ind _P ^G W$ be a parabolic induction. If $V$ is a $\varpi$-adically admissible representation of $G$ over $A$, then $W$ is a $\varpi$-adically admissible representation of $P$ over $A$.
\end{lemm}
\begin{proof} This follows from Theorem 4.4.6 in \cite{em4}.
\end{proof}

\subsection{Localisation functor}

Let $\pi$ be a supersingular representation of $\GL_2(\mb{Q}_p)$ over $k$. Recall that supersingular representations correspond to irreducible two-dimensional Galois representations under the local Langlands correspondence modulo p. See \cite{be}.

In \cite{pa1}, Paskunas has proved the following result (Proposition 5.32)
\begin{prop} We have a decomposition:
$$\Rep ^{adm} _{\GL_2(\mb{Q} _p),\xi}(\mc{O}/\varpi ^s \mc{O}) = \Rep ^{adm} _{\GL_2(\mb{Q} _p),\xi}(\mc{O}/\varpi ^s \mc{O}) _{(\pi)} \oplus \Rep ^{adm} _{\GL_2(\mb{Q} _p),\xi}(\mc{O}/\varpi ^s \mc{O}) ^{(\pi)}$$
where $\Rep ^{adm} _{\GL_2(\mb{Q} _p),\xi}(\mc{O}/\varpi ^s \mc{O})$ is the (abelian) category of smooth admissible $\mc{O}/\varpi ^s \mc{O}$-representations admitting a central character $\xi$, $\Rep ^{adm} _{\GL_2(\mb{Q} _p),\xi}(\mc{O}/\varpi ^s \mc{O}) _{(\pi)}$ (resp. $\Rep ^{adm} _{\GL_2(\mb{Q} _p),\xi}(\mc{O}/\varpi ^s \mc{O}) ^{(\pi)}$) is the subcategory of it consisting of representations $\Pi$ such that all irreducible subquotients of $\Pi$ are (resp. are not) isomorphic to $\pi$. 
\end{prop}
We denote the projection
$$\Rep ^{adm} _{\GL_2(\mb{Q} _p),\xi}(\mc{O}/\varpi ^s \mc{O}) \mapsto \Rep ^{adm} _{\GL_2(\mb{Q} _p),\xi}(\mc{O}/\varpi ^s \mc{O}) _{(\pi)}$$
by
$$V \mapsto V_{(\pi)}$$
and we refer to it as the localisation functor with respect to $\pi$. The existence of a central character follows from the work \cite{ds} for irreducible representations. In what follows, we will ignore the central character $\xi$ in our notations, though whenever we localise we mean that we take firstly the part of a representation on which the center acts via $\xi$ and then we project as above.

\section{p-adic Langlands correspondence and analytic cohomology}

In this section we show that the $p$-adic local Langlands correspondence for $\GL_2(\mb{Q}_p)$ appears in the \'etale cohomology of the Lubin-Tate tower at infinity. We also state a conjecture about the analytic cohomology of the Lubin-Tate perfectoid. 

\subsection{p-adic Langlands correspondence}

For this section we refer the reader to \cite{be} (for the Colmez functor) and \cite{pa1} (for equivalence of categories). We recall that Colmez has constructed a covariant exact functor $\mb{V}$
$$\mb{V} : \Rep _{\mc{O}} (\GL_2(\mb{Q}_p)) \rightarrow \Rep _{\mc{O}}(G_{\mb{Q}_p})$$
which sends $\mc{O}$-representations of $\GL_2(\mb{Q}_p)$ to $\mc{O}$-representations of $G _{\mb{Q}_p} = \Gal (\bar{\mb{Q}}_p / \mb{Q}_p)$. Moreover this functor is compatible with deformations and induces an equivalence of categories when restricted to appropiate sub-representations. We call the inverse of this functor the $p$-adic local Langlands correspondence and we denote it by $B(\cdot)$. For our applications we will only need the fact that for $p$-adic continuous representations $\rho : G _{\mb{Q}_p} \ra \GL_2(E)$, $B(\rho)$ is a Banach admissible $E$-representation. Furthermore, when $\rho$ is irreducible, then $B(\rho)$ is topologically irreducible.

\medskip

Let $\bar{\rho} : G _{\mb{Q}_p} \ra \GL_2(k)$ be the reduction of $\rho$ which we assume to be irreductible. Let $\pi$ be the supersingular representation of $\GL_2(\mb{Q}_p)$ over $k$ which corresponds to $\bar{\rho}$ by the mod $p$ local Langlands correspondence, that is $\mb{V} (\pi) = \bar{\rho}$. Then one knows that $B(\rho)$ is an object of the category $\Rep ^{adm} _{\GL_2(\mb{Q} _p)}(E) _{(\pi)}$ defined above.  

\subsection{\'Etale cohomology} We recall the results of Emerton on the $p$-adic completed cohomology and then we prove that certain p-adic Banach representations appear in the \'etale cohomology of $\mc{M} _{LT, \infty}$. From now on we work in the global setting. Let $\rho : G_{\mb{Q}} = \Gal (\bar{\mb{Q}}/ \mb{Q}) \ra \GL_2(E)$ be a continuous Galois representation. We assume that it is unramified outside some finite set $\Sigma = \Sigma _0 \cup \{p\}$. Moreover we assume that its reduction $\bar{\rho}$ is modular (that is, isomorphic to the reduction of a Galois representation associated to some automorphic representation on $\GL_2(\mb{Q})$) and $\bar{\rho} _p = \bar{\rho} _{| G_{\mb{Q}_p}}$ is absolutely irreducible.  

\medskip

Let us recall that we have introduced spaces $X,Y$ depending on the tame level $K^p$. We now assume that $K^p$ is unramified outside $\Sigma$. We shall factor $K^p$ as $K^p=K_{\Sigma _0} K^{\Sigma _0}$. Let $\mb{T} _{\Sigma} = \mc{O}[T_l, S_l] _{l \notin \Sigma}$ be the commutative $\mc{O}$-algebra with $T_l, S_l$ formal variables indexed by $l \notin \Sigma$. This is a standard Hecke algebra which acts on modular curves by correspondences. 

\medskip


To the modular Galois representation $\bar{\rho} : G_{\mb{Q}} \ra \GL_2(E)$ we can associate the maximal Hecke ideal $\mf{m}$ of $\mb{T} _{\Sigma}$ which is generated by $\varpi$ (uniformiser of $\mc{O}$) and elements $T_l + a_l$ and $lS_l - b_l$, where $l$ is a place of $\mb{Q}$ which does not belong to $\Sigma$, $X^2 + \bar{a} _l X^1 + \bar{b}_l$ is the characteristic polynomial of $\bar{\rho}(\Frob _l)$ and $a _l, b_l$ are any lifts of $\bar{a} _l, \bar{b} _l$ to $\mc{O}$.

\medskip

We let $\pi _{\Sigma _0}(\rho) = \otimes _{l \in \Sigma _0} \pi _l (\rho _l)$ be the tensor product of $E$-representations of $\GL_2(\mb{Q}_l)$ ($l\in \Sigma _0$) associated to $\rho _l = \rho _{|G_{\mb{Q}_l}}$ by the generic version of the $l$-adic local Langlands correspondence (see \cite{eh}). 

\medskip

We assume that $\rho$ is pro-modular in the sense of Emerton (see \cite{em2}). Let $\mf{p}$ be the prime ideal of $\mb{T} _{\Sigma}$ associated to $\rho$ (similarly as we have associated $\mf{m}$ to $\bar{\rho}$). We have an obvious inclusion $\mf{p} \subset \mf{m}$. We remark that pro-modularity is a weaker condition than modularity and it can be seen as saying that $\rho$ is a Galois representation associated to some $p$-adic Hecke eigensystem coming from the completed Hecke algebra (the projective limit over finite level Hecke algebras). Recall that we have assumed that $\bar{\rho} _p = \bar{\rho} _{| G_{\mb{Q}_p}}$ is absolutely irreducible. This permits us to state the main result of \cite{em2} as
\begin{theo} Let $\rho : G_{\mb{Q}} \ra \GL_2(E)$ be a continuous Galois representation which is pro-modular and such that $\bar{\rho}_p$ is absolutely irreducible. Then we have a $G_{\mb{Q}} \times \GL_2(\mb{Q}_p) \times \prod _{l\in \Sigma _0} \GL_2(\mb{Q}_l)$-equivariant isomorphism of Banach admissible $E$-representations.
$$H^1(Y, E)[\mf{p}] \simeq \rho \otimes _E B(\rho _p) \otimes _E \pi _{\Sigma _0}(\rho) ^{K_{\Sigma _0}}$$
\end{theo}
We recall that the cohomology group on the left is the $p$-adic completed cohomology of Emerton
$$H^1(Y, E) = \left(  \varprojlim _s \varinjlim _{K_p} H^1 _{et}(Y(K^pK_p), \mc{O} / \varpi ^s \mc{O}) \right) \otimes _{\mc{O}} E$$
where $K_p$ runs over compact open subgroups of $\GL_2(\mb{Q} _p)$.
\medskip

Let us remark that the Galois action of $G_{\mb{Q}_p}$ arises on $Y$, $X$, $\mc{M}_ {LT, \infty}$ (which we treat as adic spaces over $\Spa (C, \mc{O} _C)$) from the Galois action on the corresponding model over $\bar{\mb{Q}}_p$. 

\medskip

We also have a similar theorem for the compactification
\begin{theo} With assumptions as in the theorem above, we have an isomorphism of Banach admissible $K$-representations
$$H^1(X, E) _{\mf{m}} \simeq H^1(Y, E) _{\mf{m}}$$
In particular,
$$H^1(X, E)[\mf{p}] \simeq \rho \otimes _E B(\rho _p) \otimes _E \pi _{\Sigma _0}(\rho) ^{K_{\Sigma _0}}$$
\end{theo}
\begin{proof} We have assumed that $\bar{\rho}_p$ is absolutely irreducible and hence $\bar{\rho}$ is absolutely irreducible which implies that $\mf{m}$ is a non-Eisenstein ideal. Now the theorem follows as in the proof of Proposition 7.7.13 of \cite{em3}.
\end{proof}
We now come back to the exact sequence which we have obtained earlier
$$... \ra H^0 (X_{\sss}, \mc{O}/\varpi ^s \mc{O}) \ra H^1 _{X_{\ord}} (X, \mc{O}/\varpi ^s \mc{O}) \ra H^1 (X, \mc{O}/\varpi ^s \mc{O}) \ra H^1 (X_{\sss}, \mc{O}/\varpi ^s \mc{O}) \ra  ...$$
By Theorem 2.1.5 of \cite{em1}, we get that $H^1(X,\mc{O}/\varpi ^s \mc{O})$ is a smooth admissible $\mc{O}/\varpi ^s \mc{O}$-representation of $\GL_2(\mb{Q}_p)$. Moreover, also $H^0(X_{\sss},\mc{O}/\varpi ^s \mc{O})$ is a smooth admissible $\mc{O}/\varpi ^s \mc{O}$-representation of $\GL_2(\mb{Q}_p)$ as $X_{\sss}(K_pK^p)$ has only finite number of connected components for each $K_p$ and $K^p$. The category of smooth admissible $\mc{O}/\varpi ^s \mc{O}$-representations is a Serre subcategory of the category of smooth (not necessarily admissible) $\mc{O}/\varpi ^s \mc{O}$-representations (see Proposition 2.4.11 of \cite{em4}). Hence, as $H^1 _{X_{\ord}} (X, \mc{O}/\varpi ^s \mc{O})$ is smooth, we infer that it is also smooth admissible. By Proposition 2.4 we get that $H^1 _{X_{\ord}} (X, \mc{O}/\varpi ^s \mc{O})$ is induced from some representation $W(\mc{O}/\varpi ^s \mc{O})$ of the Borel $B(\mb{Q}_p)$. We deduce from Lemma 3.8 (Theorem 4.4.6 in \cite{em4}) that $W(\mc{O}/\varpi ^s \mc{O})$ is smooth admissible $\mc{O}/\varpi ^s \mc{O}$-representation of $B(\mb{Q}_p)$. Thus, we can apply to it Lemma 3.7. If $\pi$ is any supersingular $k$-representation of $\GL_2(\mb{Q}_p)$, it implies that
$$H^1 _{X_{\ord}} (X, \mc{O}/\varpi ^s \mc{O}) _{(\pi)} = 0$$
because by Proposition 2.3 the representation $H^1 _{X_{\ord}} (X, \mc{O}/\varpi ^s \mc{O})$ is a smooth induction of an admissible representation, hence no supersingular representation appears as a subquotient of it.


\medskip

Localising the exact sequence above at some supersingular $k$-representation $\pi$ we get an injection
$$H^1(X,\mc{O}/\varpi ^s \mc{O}) _{(\pi)} \hookrightarrow H^1 (X_{\sss}, \mc{O}/\varpi ^s \mc{O})$$
By passing to the limit with $s$ we get an injection
$$H^1(X,E) _{(\pi)} \hookrightarrow H^1 (X_{\sss}, E)$$
We can now prove our main theorem
\begin{theo} Let $\rho : G_{\mb{Q}}= \Gal (\bar{\mb{Q}} / \mb{Q}) \ra \GL_2(E)$ be a pro-modular representation. Assume that $\bar{\rho} _p = \bar{\rho} _{| G_{\mb{Q}_p}}$ is absolutely irreducible. Then we have a $\GL_2(\mb{Q}_p) \times G_{\mb{Q}_p}$-equivariant injection
$$B(\rho _p) \otimes _E \rho _p \hookrightarrow H^1(\mc{M} _{LT, \infty}, E)$$
\end{theo}
\begin{proof}
Let $\pi$ be the mod $p$ representation of $\GL_2(\mb{Q}_p)$ corresponding to $\bar{\rho}_p$ by the mod $p$ local Langlands correspondence. It is a supersingular representation by our assumption that $\bar{\rho}_p$ is absolutely irreducible. Let $\mf{p}$ be a prime ideal of $\mb{T} _{\Sigma}$ associated to $\rho$, where $\Sigma =\Sigma _0 \cup \{p\}$ is some finite set which contains $p$ and all the primes at which $\rho$ is ramified. As above we have 
$$H^1(X,E) _{(\pi)} \hookrightarrow H^1 (X_{\sss}, E)$$
and hence also
$$H^1(X,E) _{(\pi)}[\mf{p}] \hookrightarrow H^1 (X_{\sss}, E)[\mf{p}]$$
Theorem 4.2 implies that (we keep track only of $G_{\mb{Q}_p}$-action instead of $G_{\mb{Q}}$)
$$ B(\rho _p) \otimes _E \rho _p \otimes _E \pi _{\Sigma _0}(\rho)^{K_{\Sigma _0}} \hookrightarrow H^1 (X_{\sss}, E)[\mf{p}]$$
Let $K_{\Sigma _0}'$ be a compact open subgroup of $\prod _{l \in \Sigma _0} \GL_2(\mb{Q}_l)$ for which we have $\dim \pi _{\Sigma _0}(\rho) ^{K_{\Sigma _0}'} =1$ where the dimension is over $E$. Such a subgroup always exists by classical results of Casselman (see \cite{cas}). Hence we have
$$ B(\rho _p) \otimes _E \rho _p \hookrightarrow H^1 (X_{\sss}, E)[\mf{p}] ^{K_{\Sigma _0}'}$$
By Kunneth formula and Proposition 2.5 (the $p$-adic uniformisation of Rapoport-Zink) we get that 
$$H^1 (X_{\sss}, E) = \left( H^1(\mc{M}_{LT, \infty}, E) \widehat{\otimes} _E \mc{S} \right)^{D^{\times}(\mb{Q}_p)}$$
where we have denoted by $\mc{S}$ the $p$-adic quaternionic forms of level $K^p$
$$\widehat{H}^0(D^{\times}(\mb{Q}) \backslash D^{\times}(\mb{A})/K^p, E) = \left( \varprojlim _s \varinjlim _{K_p} H^0(D^{\times}(\mb{Q}) \backslash D^{\times}(\mb{A}) / K_pK^p, \mc{O} / \varpi ^s \mc{O})\right) \otimes _{\mc{O}} E$$
where $K_p$ runs over compact open subgroups of $D^{\times}(\mb{Q}_p)$. As $\GL_2(\mb{Q}_p)$ and $G_{\mb{Q}_p}$ act on $H^1(X_{\sss}, E)$ through $H^1(\mc{M}_{LT, \infty}, E)$ we conclude by the preceding discussion that 
$$ B(\rho _p) \otimes _E \rho _p \hookrightarrow H^1(\mc{M}_{LT, \infty}, E)$$
as wanted.
\end{proof}

\subsection{Cohomology with compact support}
We show that the cohomology with compact support of the Lubin-Tate tower does not contain any $p$-adic representations which reduce to mod $p$ supersingular representations. Recall we have morphisms
$$j: X _{\sss} \hookrightarrow X$$
and
$$i: X _{\ord} \ra X$$
which give an exact sequence for any \'etale sheaf $F$ on $X$
$$0 \ra j_! j^* F \ra F \ra i_* i^* F \ra 0$$
This leads to an exact sequence of the cohomology
$$... \ra H^0 (X_{\ord}, \mc{O} / \varpi ^s \mc{O}) \ra H^1 _{c} (X _{\sss}, \mc{O} / \varpi ^s \mc{O}) \ra H^1 (X, \mc{O} / \varpi ^s \mc{O}) \ra H^1 (X_{\ord}, \mc{O} / \varpi ^s \mc{O}) \ra  ...$$
Because $H^1 (X, \mc{O} / \varpi ^s \mc{O})$ is smooth admissible as a $\mc{O} / \varpi ^s \mc{O}$-representation of $\GL_2(\mb{Q}_p)$ (by the result of Emerton) and $H^0 (X_{\ord}, \mc{O} / \varpi ^s \mc{O})$ is smooth admissible as a $\mc{O} / \varpi ^s \mc{O}$-representation of $\GL_2(\mb{Q}_p)$ because at each finite level $X_{\ord}$ has a finite number of connected components, we infer that also $H^1 _{c} (X _{\sss}, \mc{O} / \varpi ^s \mc{O})$ is smooth admissible (as the category of admissible $\mc{O} / \varpi ^s \mc{O}$-representations is a Serre subcategory of smooth $\mc{O} / \varpi ^s \mc{O}$-representations). Passing to the limit with $s$, we infer that $H^1 _{c} (X _{\sss}, E)$ is Banach admissible over $E$. This means that we can localise $H^1 _{c} (X _{\sss}, E)$ at supersingular representations.

\medskip

Let $\pi$ be a supersingular $k$-representation of $\GL_2(\mb{Q}_p)$, where $k$ is the residue field of $K$. Observe that if $H^1 _{c} (X _{\sss}, E) _{(\pi)} \not =0$, then also its reduction $H^1 _{c} (X _{\sss}, k) _{(\pi)}$ would be non-zero. But Theorem 8.2 in \cite{cho} states that $H^1 _{c} (X _{\sss}, k) _{(\pi)} = 0$. Hence we get
\begin{theo} For any supersingular $k$-representation $\pi$ of $\GL_2(\mb{Q}_p)$ we have
$$H^1 _{c} (X _{\sss}, E) _{(\pi)} = 0$$
In particular
$$H^1 _{c} (\mc{M} _{LT, \infty}, E) _{(\pi)} = 0$$
\end{theo}
\begin{proof} The first part follows from the preceding discussion, the second part follows from the Rapoport-Zink uniformisation.
\end{proof}

This theorem implies that for any continuous $\rho _p : G_{\mb{Q}_p} \ra \GL_2(E)$ which has an absolutely irreducible reduction $\bar{\rho}_p: G_{\mb{Q}_p} \ra \GL_2(k)$, the $\GL_2(\mb{Q}_p)$-representation $B(\rho _p)$ associated to $\rho _p$ by the $p$-adic Local Langlands correspondence does not appear in $H^1 _{c} (\mc{M} _{LT, \infty}, E)$. Nevertheless, we believe that it appears in $H^2 _{c} (\mc{M} _{LT, \infty}, E)$, though we could not prove it.

\subsection{Analytic cohomology}

Let us explain, why we do not work with the de Rham cohomology as would the folklore conjecture suggest (to be precise: we do, but we work only with the structure sheaf as all the other differentials vanish as we show below). The reason for that is that there are no good finiteness results for de Rham cohomology of adic spaces which are not of finite type (as our Lubin-Tate perfectoid $\mc{M} _{LT, \infty}$). Moreover, it seems that the (continuous) de Rham cohomology does not suit well perfectoid spaces. Indeed, $H^i(X, \hat{\Omega} ^j _{X})$ is zero for any perfectoid space $X$ and sheaves of continuous differentials $\hat{\Omega} ^j _X$, $j>0$. We define here $\hat{\Omega} ^j _X$ locally on $\Spa(R,R^+)$ over $(K,K^+)$ by firstly defining
$$\hat{\Omega} ^j _{R^+ / K^+} = \varprojlim _n \Omega ^j _{(R^+ / p^n) / (K^+/p^n)}$$
and then $\hat{\Omega} ^j _{\Spa(R,R^+)} = \hat{\Omega} ^j _{R / K} = \hat{\Omega} ^j _{R^+ / K^+} [1/p]$. Thus, it is enough to prove the statement for affinoid perfectoids $X = \Spa(R, R^+)$. We can further reduce ourselves to the case $i = 0$ by using the Cech complex associated to some rational covering of $X$ (which will be a covering by affinoid perfectoids by Corollary 6.8 of \cite{sch1}). Hence, we have to show that global sections of $\hat{\Omega} _X ^{j}$ are zero. It suffices to show that $\hat{\Omega} _{R^+/ K^+} ^{j}$ is almost zero. This follows from induction, as for $n=1$ the sheaf $\Omega ^j _{(R^+ / p) / (K^+/p)}$ is identically zero, and for $n>1$ we conclude using an exact sequence, as in the proof of Theorem 5.10 of \cite{sch1}:
$$0 \ra R^+/p \ra R^+ / p^n \ra R^+/p^{n-1} \ra 0$$ 
Let us remark that this reasoning also implies that sheaves $\hat{\Omega} ^j _X$ are zero on a perfectoid space $X$ for $j>0$. It is enough to check it at stalks where we have $\hat{\Omega} ^j _{X,x} = \varinjlim _{x \in U} \hat{\Omega} ^j _{X}(U)$ and $U$ runs over rational affinoid subsets of $X$ containing $x$. As such subsets are perfectoid (Corollary 6.8 of \cite{sch1}) we have $\hat{\Omega} ^j _X (U) = 0$ and hence the result. 
  


\medskip

As $\mc{M} _{LT, \infty}$ is a perfectoid space by \cite{sw}, the above reasoning applies, showing that de Rham cohomology of $\mc{M} _{LT, \infty}$ reduces to the study of the cohomology with values in the structure sheaf. This is exactly the analytic cohomology we consider. By using recent results of Scholze, it seems natural to work with the analytic cohomology (i.e. topology defined by open subsets). We review this below. We believe that the 'folklore conjecture' should be understood as the statement that the $p$-adic local Langlands correspondence appears in the analytic cohomology of the appropiate Rapoport-Zink space at infinity. We also remark that the same applies to Shimura varieties at the infinite level, which are perfectoid spaces by \cite{sch3}.

\medskip

If $Z$ is any adic space, we denote by $Z_{an}$ its analytic topos which arises from the topology of open subsets. For any (coherent) sheaf $\mc{F}$ on $Z$, we write $H^i _{an}(Z, \mc{F})$ for the $i$-th cohomology group of $Z _{an}$ with values in $\mc{F}$.

\medskip

By Theorem IV.2.1 of \cite{sch3} (where we pass to the limit with $\mb{Z} / p^n \mb{Z}$ and then use the reasoning from the proof of Theorem 3.20 in \cite{sch4} to descent from the pro-\'etale site to the \'etale site) we have an isomorphism
$$H^1(X, E) \widehat{\otimes} _E C \simeq H^1 _{an} (X, \mc{O}_X)$$
which is $\GL_2(\mb{Q}_p)$-equivariant and also equivariant with respect to the Hecke action of $\mb{T}_{\Sigma}$. 


\medskip

If we were to use the same reasoning as for the $p$-adic completed cohomology (i.e. some exact sequence of analytic cohomology and localisation at a supersingular representation) to show that the $p$-adic local Langlands correspondence appears in the analytic cohomology of the Lubin-Tate tower at infinity, then we would have to start by proving admissibility of the cohomology groups. Unfortunately, this is not true. By the comparison theorem of Scholze we get that $H^1 _{an} (X, \mc{O}_X)$ is a Banach admissible $E$-represention, but $H^0 _{an} (X_{\sss}, \mc{O} _{X_{\sss}})$ is not admissible (and it is not even clear whether it is a Banach space). In order to prove that, it is enough to prove it for $H^0 _{an} (\mc{M} _{LT, \infty}, \mc{O} _{\mc{M}_{LT, \infty}})$ by the $p$-adic uniformisation of Rapoport-Zink.

\begin{prop}
The $\GL_2(\mb{Q}_p)$-representation $H^0 _{an} (\mc{M} _{LT, \infty}, \mc{O} _{\mc{M}_{LT, \infty}})$ is not admissible.
\end{prop}
\begin{proof}
In Section 2 (see especially 2.10) in \cite{we}, Weinstein gives an explicit description of the geometrically connected components of $\mc{M} _{LT, \infty}$. Each of them is isomorphic to $\Spa (A \otimes _{\mc{O}_{K_{\infty}}} C, A \otimes _{\mc{O} _{K_{\infty}}} \mc{O}_C)$, where $K_{\infty}$ is the Lubin-Tate extension of $\mb{Q}_p$ (see Section 2.3 of \cite{we}; we fix an embedding $K_{\infty} \hookrightarrow C$) and $A$ is a perfectoid $K_{\infty}$-algebra with a tilt (Corollary 2.9.11 of \cite{we})
$$A ^{\flat} \simeq \bar{\mb{F}}_p [[ X_1 ^{1/ p^{\infty}}, X_2 ^{1/p^{\infty}}]]$$
Hence, in $H^0 _{an} (\mc{M} _{LT, \infty}, \mc{O} _{\mc{M}_{LT, \infty}})$ appears $A \otimes _{\mc{O}_{K_{\infty}}} C$ (and in fact much more as this is the set of all unbouded funtions on the Lubin-Tate perfectoid). We have an action of $\GL_2(\mb{Z}_p)$ on $A$. Let $K$ be any compact open subgroup of $\GL_2(\mb{Z}_p)$. If $H^0 _{an} (\mc{M} _{LT, \infty}, \mc{O} _{\mc{M}_{LT, \infty}})$ were admissible, then in particular for the lattice $A \otimes _{\mc{O}_{K_{\infty}}} \mc{O} _C$ in $A \otimes _{\mc{O}_{K_{\infty}}} C$, the reduction of $K$-invariants $(A \otimes _{\mc{O}_{K_{\infty}}} \bar{\mb{F}}_p)^K$ would be of finite dimension over $\bar{\mb{F}}_p$ (by the very definition, see Definition 2.7.1 of \cite{em4}). This is not possible. Indeed, observe that $A^K$ contains (and probably equals to but we do not need it) the ring of integral analytic functions on the Lubin-Tate space of $K$-level, which is a finite ring over the ring $\mc{O} _{C}[[X_1,X_2]]$ of power-series over $\mc{O}_C$.
\end{proof}

This means that we cannot use the localisation functor and deduce our result from the global results of Emerton. Hence, for now, we can only state a conjecture, which we believe to be a correct version of the folklore conjecture.

\begin{conj} Let $\rho _p : G_{\mb{Q}_p} \ra \GL_2(E)$ be a continuous de Rham Galois representation. Then, there is a non-zero $\GL_2(\mb{Q}_p)$-equivariant injection
$$B(\rho_p) \hookrightarrow H^1 _{an} (\mc{M} _{LT, \infty}, \mc{O} _{\mc{M}_{LT, \infty}}) $$
\end{conj}

Observe that in fact we can state a similar conjecture for $H^0 _{an} (\mc{M} _{LT, \infty}, \mc{O} _{\mc{M}_{LT, \infty}})$ instead of $H^1 _{an}$. A priori, it is not clear which one should be true or whether both are. The advantage of working with $H^0 _{an}$ should be the fact that it is quite explicit by the work of Weinstein. 

\medskip

Christophe Breuil has informed us that a similar conjecture was made by him and Matthias Strauch in 2006 (unpublished note). The difference was that on the left side they considered the locally analytic vectors of $B(\rho_p)$ while on the right side they had a cohomology of the Drinfeld tower at some finite level. Results toward this conjecture for special series appear in \cite{br}.

\medskip

We believe that there is also a more refined version of the folklore conjecture which truly realizes the $p$-adic local Langlands correspondence in the sense that in the analytic cohomology of the Lubin-Tate perfectoid should appear a tensor product of $B(\rho_p)$ with the associated $(\phi, \Gamma)$-module of $\rho _p$. We do not make precise here what kind of $(\phi, \Gamma)$-modules we consider and how the appropiate Robba ring acts on the Lubin-Tate perfectoid. We shall come back to those issues elsewhere.

\subsection{Final remarks}

Observe that our proof of Theorem 4.3 depends on the global data as we have to start with a global pro-modular Galois representation $\rho$. As our result is completely local, it is natural to ask whether the same thing holds for any absolutely irreducible Galois representation $\rho _p$ of $G_{\mb{Q}_p}$ which is not necessarily a restriction of some global $\rho$ (as in Conjecture 4.6).

\medskip

Another natural problem is to try to prove Theorem 4.3 without assuming that $\bar{\rho}_p$ is absolutely irreducible. This would require a more careful study of the cohomology of the ordinary locus.

\medskip

The most pertaining problem is whether one can reconstruct $B(\rho _p)$ from either the $p$-adic completed or the analytic cohomology of the Lubin-Tate tower and hence give a different proof of the $p$-adic local Langlands correspondence. This might be useful in trying to prove the existence of the $p$-adic correspondence for groups other than $\GL_2(\mb{Q}_p)$ as well as Theorem 4.3 for Galois representations $\rho _p$ not necessarily coming from global Galois representations.

\end{document}